\documentclass[a4paper]{amsart}
\usepackage{amsfonts}
\usepackage{amsthm}
\usepackage{amsmath}
\usepackage{amscd}
\usepackage[latin2]{inputenc}
\usepackage{t1enc}
\usepackage[mathscr]{eucal}
\usepackage{indentfirst}
\usepackage{graphicx}
\usepackage{graphics}
\usepackage{epic}
\numberwithin{equation}{section}
\usepackage[margin=2.9cm]{geometry}
\usepackage{epstopdf}
\usepackage{tikz}
\usepackage{amssymb}
\usepackage{array}
\usepackage{enumerate}
\usepackage{xy}
\usepackage{pdfsync}

\theoremstyle{plain}
\newtheorem{theorem}{Theorem}[section]

\newtheorem{lemma}[theorem]{Lemma}

\newtheorem{proposition}[theorem]{Proposition}

\newtheorem{algorithm}[theorem]{Algorithm}
 \theoremstyle{definition}

\newtheorem{example}[theorem]{Example}

\DeclareMathOperator\R{\mathbb{R}}

\DeclareMathOperator\E{\mathbb{E}}
\DeclareMathOperator\I{\mathbb{I}}
\DeclareMathOperator\Pb{\mathbb{P}}
\DeclareMathOperator\Var{\text{Var}}

\allowdisplaybreaks

\newcommand{\mc}{\mathcal}

\newcommand{\abs}[1]{\left| #1 \right|}

\newcommand{\floor}[1]{\lfloor #1 \rfloor}

\title[Importance Sampling for Perfect Matching Problems]{Theoretical Analysis of Sequential Importance Sampling Algorithms for a Class of Perfect Matching Problems}

\author[A. Tsao]{Andy Tsao}

\address{Stanford University \\ Department of Statistics \\
Stanford CA 94305}

\email{andytsao@stanford.edu}


\begin{document}
\maketitle
\begin{abstract}
  This paper analyzes the performance of sequential importance sampling algorithms for estimating the number of perfect matchings in bipartite graphs. Precise bounds on the number of samples required to yield an accurate estimate are derived. In doing so, moments of permutation statistics are computed using generating functions and nonstandard limit theorems are derived by expressing perfect matchings as a time-inhomogeneous Markov chain.
\end{abstract}
\section{Introduction}\label{sec:introduction}
\noindent Sequential importance sampling is a technique for estimating the expected value of a given function with respect to a probability measure $\nu$ using a random sample from a different probability measure $\mu$. It is widely used to evaluate otherwise intractable counting and statistical problems. This work examines the performance of sequential importance sampling on counting the number of perfect matchings in bipartite graphs. This problem can also be formulated equivalently as counting the number of permutations with positions restricted by a binary matrix.

In importance sampling, one uses a simple measure $\mu$ to obtain information about a more complicated measure $\nu$. In~\cite{cd18}, Chatterjee and Diaconis show that if $\log(d\nu/d\mu)$ is concentrated about its mean, then a sample size of roughly $e^L$ from $\mu$ is necessary and sufficient, where $L$ denotes the Kullback-Leibler divergence between $\nu$ and $\mu$. The objective for this work will be to prove limit theorems and control the tail probabilities of the quantity $\log(d\nu/d\mu)$ in the context of restricted permutations.

The remainder of this section reviews the relevant literature on matchings, restricted permutations, and sequential importance sampling. Section 2 introduces a sequential algorithm for sampling a specific type of restricted permutation. Section 3 summarizes the empirical results from using this algorithm. Sections 4, 5, and 6 analyze the moments and limiting distribution of certain statistics of restricted permutations and uses them to give a bound on the required sample size for importance sampling to give accurate results.
\subsection{Bipartite matchings}
Let $[n] = \{1, 2, \ldots, n\}$ and $[n'] = \{1', 2', \ldots, n'\}$ be two disjoint sets. A bipartite graph $G = ([n], [n'], E)$ is specified by a set of undirected edges $E = \{(i_1, i_1'), \ldots, (i_e, i_e')\}$. For example, when $n = 3$ the graph might appear as shown in Figure~\ref{fig:bipartite}.

\begin{figure}[h]
  \begin{center}
    \begin{tikzpicture}
      \draw[black, thick] (0, -1)--(2, -1);
      \draw[black, thick] (0, 0)--(2, 0);
      \draw[black, thick] (0, 1)--(2, 1);
      \draw[black, thick] (0, -1)--(2, 0);
      \draw[black, thick] (0, 0)--(2, -1);
      \draw[black, thick] (0, 1)--(2, 0);
      \draw[black, thick] (0, -1)--(2, 1);
      \filldraw[black] (0, 1) circle (2pt) node[anchor=east] {$1$};
      \filldraw[black] (2, 1) circle (2pt) node[anchor=west] {$1'$};
      \filldraw[black] (0, 0) circle (2pt) node[anchor=east] {$2$};
      \filldraw[black] (2, 0) circle (2pt) node[anchor=west] {$2'$};
      \filldraw[black] (0, -1) circle (2pt) node[anchor=east] {$3$};
      \filldraw[black] (2, -1) circle (2pt) node[anchor=west] {$3'$};
    \end{tikzpicture}
    \caption{A bipartite graph with $n = m = 3$}\label{fig:bipartite}
  \end{center}
\end{figure}
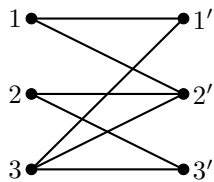

A \emph{matching} in $G$ is a set of vertex-disjoint edges. Thus $\{(1, 1'), (2, 3')\}$ is a matching in Figure~\ref{fig:bipartite}, as is the empty set. A \emph{perfect matching} in $G$ is a matching with $n$ edges. For example, the perfect matchings in Figure~\ref{fig:bipartite} are $\{(1, 1'), (2, 2'), (3, 3')\}$, $\{(1, 2'), (2, 3'), (3, 1')\}$, and $\{(1, 1'), (2, 3'), (3, 2')\}$. $\mc{M}(G)$ will be used to denote the set of perfect matchings of a graph $G$.

Matching theory is a large research area, particularly recently with ride share and organ matching applications. See~\cite{lp09} for a book-length treatment.
\subsection{Restricted permutations}\label{sec:restricted}
Given a bipartite graph $G([n], [n'], E)$, let $A_G$ denote its adjacency matrix; that is, $A_G(i, j) = \I\{(i, j') \in E\}$. The perfects matching of $G$ correspond to a subset $S_G \in S_n$ of permutations $\pi$ satisfying $A(i, \pi_i) = 1$ for all $i$. For example, if $G$ is the graph in Figure~\ref{fig:bipartite},
\[A_G = \begin{pmatrix}
  1 & 1 & 0 \\
  0 & 1 & 1 \\
  1 & 1 & 1
\end{pmatrix}\]
and $S_G = \{(123), (231), (132)\}$.

Of particular consideration are the matrices
\[A_G(i, j) = \begin{cases}
  1 &\mbox{if~} 1 \le i \le n, 1 \le j \le m, -s \le j - i \le t \\
  0 &\mbox{otherwise}
\end{cases}\]
where $s, t \ge 1$. $G$ is called the \emph{type-$(s, t)$ graph}, and the elements of $\mc{S}_G$ are called type-$(s, t)$ permutations, denoted by $\mc{F}_{n, s, t}$.

The special case of $s = t = 1$ corresponds to the Fibonacci permutations, so named because $\abs{\mc{F}_{n, 1, 1}} = F_n$, where $F_n$ is the $n^{\text{th}}$ Fibonacci number. Other well-studied cases include $\mc{F}_{n, t, 1}$ and $\mc{F}_{n, t, t}$, which are sometimes called $t$-Fibonacci permutations and distance-$t$ permutations, respectively.

Type-$(s, t)$ graphs serve as benchmarks for both numerical and theoretical purposes, and they offer challenging open problems, despite being extensively studied (\cite{cdg18},~\cite{dk19},~\cite{dgh01}). Furthermore, despite their apparent structure, they are a good approximation to graphs appearing in real datasets (see, for instance, the red shift data in~\cite{ep92}).
\subsection{Importance sampling}
Let $\mu$ and $\nu$ be two probability measures on a set $\mc{X}$ equipped with some $\sigma$-algebra. Suppose $\nu \ll \mu$, and let $\rho$ denote the density $\frac{d\nu}{d\mu}$. To estimate the quantity
\[I(f) := \int_{\mc{X}} f(y) d\nu(y) = \E_\nu f(Y)\]
using an iid sample $X_1, X_2, \ldots$ with distribution $\mu$, the \emph{importance sampling estimate} of $I(f)$ is given by
\[I_N(f) := \frac{1}{N} \sum_{i = 1}^N f(X_i) \rho(X_i).\]

The number of perfect matchings of a balanced bipartite graph $G = ([n], [n'], E)$ can be estimated using importance sampling. Taking $\nu$ to be the uniform measure, $\mu$ to be any other measure on perfect matchings, and $f = \abs{\mc{M}(G)}$, the quantity $I(f) = \abs{\mc{M}(G)}$ has the importance sampling estimate
\[I_N(f) = \frac{1}{N} \sum_{i = 1}^N \abs{\mc{M}(G)} \frac{d\nu}{d\mu}(X_i) = \frac{1}{N} \sum_{i = 1}^N \mu(X_i)^{-1},\]
where $X_1, \ldots, X_N$ are perfect matchings with distribution $\mu$.

In applications of importance sampling, the measure $\mu$ is typically chosen so that $X_1, \ldots, X_N$ are easy to sample. Diaconis~\cite{dia18} proposed the following sequential algorithm for generating perfect matchings in a bipartite graph:
\begin{algorithm}\label{alg:sequential}
  Let $v_1, \ldots, v_n$ be an enumeration of the vertices in $[n]$, and let $\pi_0$ be the empty matching. Proceeding in the order $i = 1, 2, \ldots, n$:
  \begin{itemize}
    \item
      Check each edge coming out of $v_i$ to see if its removal, and the subsequent removal of the adjacent vertices, leaves a graph allowing a perfect matching. Let $J_i$ be the set of available edges.
    \item
      Pick $e \in J_i$ uniformly. Let $\pi_i = \pi_{i = 1} \cup \{e\}$.
    \item
      This generates a random matching $\pi_n$ with probability
      \[\mu(\pi_n) = \prod_{i = 1}^n \abs{J_i}^{-1}.\]
  \end{itemize}
\end{algorithm}
It will be useful in this paper to form an equivalence between the sequence $\{J_i\}_{i = 1}^n$ and the resulting permutation $\pi$ in Algorithm~\ref{alg:sequential}. Indeed, a bijection exists between the two quantities:
\begin{itemize}
  \item
    From a permutation $\pi$, the sequence $J_1, \ldots, J_n$ is obtained by setting $J_i = E(v_i) \backslash \{\pi(v_1), \ldots, \pi(v_{i - 1})\}$, where $E(v_i)$ denotes the vertices adjacent to $v_i$.
  \item
    Conversely, a sequence $J_1, \ldots, J_n$ yields the permutation $\pi$ satisfying $\pi_i = E(v_i) \backslash \bigcup_{j = i + 1}^n J_j$.
\end{itemize}
Unless otherwise stated, the analysis of Algorithm~\ref{alg:sequential} will be of the \emph{top-down} order; that is, $v_i = i$ for all $1 \le i \le n$.

The procedure for checking if an arbitrary bipartite graph has a perfect matching is polynomial in $n$. However, this step can be done in constant time for type-$(s, t)$ graphs.
\begin{proposition}\label{prop:alg_success}
  Let $G = ([n], [n'], E)$ be a type-$(s, t)$ bipartite graph. Suppose that the vertices $\{1, 2, \ldots, i - 1\}$ have been matched by Algorithm~\ref{alg:sequential}. If $(i - s)'$ has not yet been matched, then $J_i = \{(i - s)'\}$. Otherwise, $J_i$ contains all remaining edges incident to $i$.
\end{proposition}
Chatterjee and Diaconis~\cite{cd18} argue that the distribution of $\rho(Y) = \frac{d\nu}{d\mu}(Y)$ is key to determine the necessary and sufficient sample size for $I_n(f)$ to yield a good estimate of $I(f)$. In particular, they proved an upper bound on the necessary sample size that is directly related to the tails of $\log \rho(Y)$. Taking $\nu$ and $\mu$ to be the uniform distribution on matchings and the sampling distribution of Algorithm~\ref{alg:sequential}, respectively, yields
\[\log \rho(Y) = \log \frac{1}{\abs{\mc{M}(G)} \mu(Y)} = -\log \abs{\mc{M}(G)} - \log \mu(Y).\]
A main contribution of this work is the distributional analysis of the quantity $\log \mu(Y)$ under the uniform distibution on matchings for several classes of bipartite graphs.

\section{Related Work}\label{sec:related}
\noindent Restricted permutations appear in problems related to independence testing. One observes paired data $(X_1, Y_1), \ldots, (X_n, Y_n) \in \mc{X} \times \mc{Y}$ drawn from a joint distribution $\mc{P}$, with marginals $\mc{P}^1$ and $\mc{P}^2$. For simplicity, assume that the $X_i$'s and $Y_i$'s are all distinct. Suppose further that for each $x \in \mc{X}$ there is a known set $I(x)$ such that the pair $(X, Y)$ can be observed if and only if $Y \in I(X)$.

Suppose the goal is to test if $\mc{P} = \mc{P}^1 \times \mc{P}^2$. If $I(x) = \mc{Y}$ for all $x \in \mc{X}$, then classical theory (see, e.g.~\cite{hoe48},~\cite{bkr61},~\cite{bic69},~\cite{rom89}) tells us that under mild regularity conditions, a permutation test gives an asymptotically consistent locally most powerful test of independence. That is, let $(X_{(1)}, \ldots, X_{(n)})$ and $(Y^{(1)}, \ldots, Y^{(n)})$ be the rank-orderings of the $\{X_i\}$ and $\{Y_i\}$, respectively, and define the permutation $\pi$ to be such that $Y_{(i)} = Y^{(\pi(i))}$ for all $i \in \{1, 2, \ldots, n\}$. $X$ and $Y$ then pass the permutation test if $\pi$ looks like it came from a random draw from $S_n$.

The setting where $I(x)$ is a proper subset of $\mc{Y}$ can be modeled as a permutation test on a set of permutations with restricted positions. In this case, it is necessary to characterize a random draw from $S_{n, A_{n, n}} \subset S_n$, where $A$ is a restriction function as defined in Section~\ref{sec:restricted}. This is equivalent to evaluating the permanent of $A_{n, n}$.

Evaluating the permanent of a $\{0, 1\}$ matrix is a celebrated problem in complexity theory and was used as the first example of a \#P-complete problem by Valiant~\cite{val79}. However, while exact evaluation remains an intractable problem, efficient approximation algorithms sometimes exist.

Diaconis et.\ al.~\cite{dgh01} proposed the \emph{switch chain} for sampling perfect matchings from a balanced bipartite graph $G = ([n] \cup [n]', E)$ almost uniformly at random. The largest class of graphs for which this chain is ergodic is the class of chordal bipartite graphs. In~\cite{djm17}, Dyer et.\ al.\ examine increasingly restricted graph classes and determine that the switch chain mixes in time $O(n^7 \log n)$ for monotone graphs. This bound was later improved by Blumberg~\cite{blumthesis} to $O(n^4)$ for graphs with bounded interval restrictions.

Diaconis and Kolesnik~\cite{dk19} analyze Algorithm~\ref{alg:sequential} for $t$-Fibonacci and distance-2 matchings. They were able to prove the asymptotic normality of $\log \rho(Y)$ using a distributional recurrence Central Limit Theorem from the computer science literature. Using generating functions, Chung et.\ al.~\cite{cdg18} were also able to compute precise asymptotics for the mean and variance of $\log \rho(Y)$ for the cases $t = 1$ and $(s, t) = (2, 2)$. Moments for more general $s, t$ are open. Finally,~\cite{dk19} also analyzes two additional algorithms for $t$-Fibonacci matchings: the \emph{random} order algorithm, where $(v_1, \ldots, v_n)$ is a random permutation of $[n]$, and the \emph{greedy} order algorithm, where at each step, the smallest unmatched index $i$ is matched amongst those indices $i$ with the maximal number of remaining choices for $\pi(i)$. Central limit theorems with precise asymptotics are also derived for both of these algorithms.

\section{Results}\label{sec:results}
\noindent The contributions of this work are threefold. First, an exact formula is provided for the sampling distribution $\mu(\pi)$ of Algorithm~\ref{alg:sequential}.
\begin{proposition}\label{prop:st_sampling_formula}
  For each $y = 1, 2, \ldots, n$, Let $X_i = \I\{\pi(i) = i - s\}$and $Y_i = \min(t + 1, n - i)$. Then,
  \begin{equation}\label{eq:mu_formula}
    \mu(\pi) = \frac{(t + 1)^{t - n}}{t!} \cdot \prod_{i = 1}^n Y_i^{X_i}.
  \end{equation}
  In particular, there exist constants $c_1 = c_1(s, t)$ and $c_2 = c_2(s, t)$ such that
  \[c_1 (t + 1)^{\theta(\pi) - n} \le \mu(\pi) \le c_2 (t + 1)^{\theta(\pi) - n},\]
  where $\theta(\pi) = \sum_{i = 1}^n X_i = \abs{\{i: \pi(i) = i - s\}}$.
\end{proposition}
Next, this work extends the results of Diaconis and Kolesnik in~\cite{dk19}. The following distributional result holds for arbitrary positive integers $s$ and $t$:
\begin{theorem}\label{thm:importance_sample_size}
  Let $G = ([n], [n'], E)$ be the bipartite graph with type-$(s, t)$ restriction, and let $\mu(\pi)$ be the sampling distribution of Algorithm~\ref{alg:sequential} when $v_i = i$ for $1 \le i \le n$. Then, there exist positive constants $c_1, c_2$ such that
  \begin{align}
    \E_\nu \log \rho(Y) + \log \abs{\mc{M}(G)} &= c_1 n + o(n) \label{eq:exp_logrho}\\
    \Var_\nu \log \rho(Y) &= c_2 n + o(n) \label{eq:var_logrho}
  \end{align}
  Furthermore, as $n \rightarrow \infty$,
  \[\frac{\log \rho(Y) - \E_\nu \log \rho(Y)}{\sqrt{\Var_\nu \log \rho(Y)}} \overset{d}\rightarrow N(0, 1).\]
\end{theorem}
The implication of Theorem~\ref{thm:importance_sample_size} and the result in~\cite{cd18} is that Algorithm~\ref{alg:sequential} converges after $N_{conv} \approx \exp(c_1 n + \sqrt{c_2 n})$ samples. Since it takes time $O(n)$ to generate a single perfect matching, the aggregate runtime of sequential importance sampling is $O(N_{conv}n) = O(n \exp(c_1 n + \sqrt{c_2 n}))$. At first glance, this is clearly inferior to the $O(n^7 \log n)$ runtime for monotone graphs, given in~\cite{djm17}, or the $O(n^4)$ runtime for graphs with bounded interval restrictions, given in~\cite{blumthesis}. However, it turns out that Algorithm~\ref{alg:sequential} has some merit, as the constants $c_1$ and $c_2$ are often very small. As can be seen in Table~\ref{tab:runtime_comparison}, $n$ needs to be quite large to justify using either MCMC algorithm over importance sampling.
\begin{table}[h]
  \centering
  \begin{tabular}{c|c|c}
    $(s, t)$ & $N_1$ & $N_2$ \\ \hline
    $(2, 1)$ & 1035 & 2592 \\
    $(3, 1)$ & 2049 & 5018 \\
    $(4, 1)$ & 4332 & 10415 \\
    $(5, 1)$ & 9319 & 22071 \\
    $(6, 1)$ & 20115 & 47056 \\
    $(7, 1)$ & 43358 & 100399 \\
    $(3, 2)$ & 308 & 804
  \end{tabular}
  \caption{Comparison of importance sampling and the switch chain. $N_1$ and $N_2$ are the sample sizes below which importance sampling outperforms switch chain bounds of $O(n^4)$ for bounded interval restrictions and $O(n^7 \log n)$ for monotone bipartite graphs, respectively.}
  \label{tab:runtime_comparison}
\end{table}
\subsection{Optimal sampling probabilities}
\noindent In addition to proving Theorem~\ref{thm:importance_sample_size}, this paper also considers a modification of Algorithm~\ref{alg:sequential}, where edges are picked from the available set nonuniformly at each step. More precisely, Let $P_{j, J}$ be a family of probability distributions, indexed by $j \in [n]$ and $J \subseteq [n']$.
\begin{algorithm}[Nonuniform sequential algorithm]\label{alg:sequential_opt}
  Let $v_1, \ldots, v_n$ be an enumeration of the vertices in $[n]$. Beginning at $v_1$ and proceeding in order:
  \begin{itemize}
    \item
      Check each edge coming out of $v_1$ to see if its removal, and the subsequent removal of the adjacent vertices, leaves a graph allowing a perfect matching. Let $J_1$ be the set of available edges. Pick $e \in J_1$ according to the distribution $P_{1, J_1}$ and delete this edge.
    \item
      Repeat with $v_2$ by forming $J_2$ and sampling from $P_{2, J_2}$, and continue until a perfect matching is found.
    \item
      This generates a random matching $\pi$ with probability
      \[\mu^*(\pi) = \prod_{i = 1}^n P_{i, J_i}(\pi(i)).\]
  \end{itemize}
\end{algorithm}
It is immediately clear that choosing $P_{j, J}$ to be the distribution of $\pi(j)$ conditioned on $\pi(1), \ldots, \pi(j - 1)$ makes $\mu^*$ the uniform distribution on allowed matchings. However, explicitly computing these conditional distributions is impractical for all but the simplest bipartite graphs.

Diaconis and Kolesnik~\cite{dk19} analyze the top-down version of Algorithm~\ref{alg:sequential_opt} (where $v_i = i$ for all $i$) for Fibonacci, 2-Fibonacci, and distance-2 graphs. They show that, for these graphs, it is possible to choose $P_{j, J}$ from a much smaller family of distributions such that Algorithm~\ref{alg:sequential_opt} yields a sampling distribution with bounded derivative $\frac{d\nu}{d\mu^*}$. An example of their results for Fibonacci graphs is as follows:
\begin{proposition}\label{prop:opt_bounded_deriv}
  For a set of two integers $J = \{j_1, j_2\}$ with $j_1 < j_2$, let $Q_J$ be the distribution that assigns mass $1/\varphi$ to $j_1$ and $1/\varphi^2$ to $j_2$. Let $P_{j, J} = Q_J$ whenever $\abs{J} = 2$. Then, the resulting sampling distribution $\mu^*$ has bounded derivative $\frac{d\nu}{d\mu^*}$ with respect to the uniform distribution $\nu$.
\end{proposition}
A direct consequence of this type of result is that importance sampling using the distribution $\mu^*$ converges after a bounded number of samples. The final contribution of this paper will be the construction of a simple family $P_{j, J}$ for type-$(s, t)$ graphs such that $\log \rho(Y)$ is bounded.

The remainder of this paper is organized as follows. Section~\ref{sec:markov} constructs a bijection between matchings of type-$(s, t)$ graphs and Markovian sequences and uses it to prove Theorem~\ref{thm:importance_sample_size} and derive ``almost-perfect'' sampling probabilities. Section~\ref{sec:moments} computes moments of $\log \rho(Y)$ using generating functions. Conclusions and ideas for further research are given in Section~\ref{sec:future}. Finally, Section~\ref{sec:proofs} contains the derivations of all unproven claims throughout the chapter.

\section{Restricted permutations as Markov chains}\label{sec:markov}
\noindent A key observation for the analysis of Algorithm~\ref{alg:sequential} is that a uniform draw from $\mc{M}_{n, s, t}$ can be expressed as a time-inhomogeneous Markov chain, where the transition matrices have entries that are bounded by functions of $s$ and $t$. Distributional limits of functions of these Markov chains were first studied by Dobrushin~\cite{dob56} and later refined in~\cite{sv05} and~\cite{pel12}. The following result is due to Peligrad~\cite{pel12} and establishes conditions on the maximal correlation coefficient between adjacent states $X_i$ and $X_{i + 1}$ under which a central limit theorem would hold.
\begin{theorem}[\cite{pel12}, Theorem 1]\label{thm:TIMC_clt}
  Let $X_{n, 1}, \ldots, X_{n, n} \in \mc{X}$ be a time-inhomogeneous Markov chain. Let $\rho(\cdot, \cdot)$ denote the maximal correlation function; that is, for $\sigma$-algebras $\mc{F}_1, \mc{F}_2$,
  \[\rho = \sup_{f, g} \E(fg),\]
  where $f$ and $g$ are functions with mean zero and variance one which are measurable with respect to $\mc{F}_1$ and $\mc{F}_2$, respectively. Define
  \[\lambda_n = \min_{1 \le s \le n - 1} [1 - \rho(\sigma(X_{n, s}), \sigma(X_{n, s + 1}))],\]
  Let $Y_{n, i} = f_{n, i}(X_{n, i})$, where $(f_{n, i})_{1 \le i \le n}$ are real-valued functions on $\mc{X}$. Denote by $\mu_n$ and $\sigma^2_n$, respectively, the mean and variance of $\sum_{i = 1}^n Y_{n, i}$. Suppose
  \begin{equation}\label{eq:CLT_conditions_1}
    \max_{1 \le i \le n} \abs{Y_{n, i}} \le C_n \text{~a.s.}
  \end{equation}
  and
  \begin{equation}\label{eq:CLT_conditions_2}
    \frac{C_n(1 + \abs{\ln(\lambda_n)})}{\lambda_n \sigma_n} \rightarrow 0 \text{~as~} n \rightarrow \infty.
  \end{equation}
  Then
  \begin{equation}\label{eq:clt_statement}
    \frac{\sum_{i = 1}^n Y_{n, i} - \mu_n}{\sigma_n} \overset{d}\rightarrow N(0, 1).
  \end{equation}
\end{theorem}
In order to apply this result, a Markov chain representation of type-$(s, t)$ permutations must be constructed.
\subsection{Type-$(s, t)$ sequences}\label{subsec:st_sequence}
The state space for the Markov chain are sequences of integers $x = (n_1, \ldots, n_t)$ satisfying $s \ge n_1 \ge n_2 \cdots \ge n_t$. Let $\mc{X}_{s, t}$ denote the set of all such sequences. Further, let $\mc{X}'_{s, t} \subset \mc{X}_{s, t}$ denote the subset of sequences with $n_1 = s$.

For each each state $\mc{X}_{s, t} \ni x = (n_1, \ldots, n_t)$, let $A_{n, x}$ denote the binary matrix satisfying the following conditions:
\begin{itemize}
  \item
    If $j - i < -s$ or $j - i > t$, then $A_{n, x}(i, j) = 0$
  \item
    For all $i \le t$, if $j - i < -s + n_i$, then $A_{n, x}(i, j) = 0$
  \item
    For all other pairs $(i, j)$, $A_{n, x}(i, j) = 1$
\end{itemize}
For example, for $n = 8$, $s = 3$, $t = 2$, and $x = (2, 1)$,
\[A_{n, x} = \begin{pmatrix}
  1 & 1 & 1 & 0 & 0 & 0 & 0 & 0 \\
  1 & 1 & 1 & 1 & 0 & 0 & 0 & 0 \\
  {\color{red} 0} & 1 & 1 & 1 & 1 & 0 & 0 & 0 \\
  {\color{red} 0} & 1 & 1 & 1 & 1 & 1 & 0 & 0 \\
  0 & {\color{red} 0} & 1 & 1 & 1 & 1 & 1 & 0 \\
  0 & 0 & 1 & 1 & 1 & 1 & 1 & 1 \\
  0 & 0 & 0 & 1 & 1 & 1 & 1 & 1 \\
  0 & 0 & 0 & 0 & 1 & 1 & 1 & 1
\end{pmatrix}\]
Note that, compared to $A_{n, (0, 0)}$, there are two extra zeroes in the first column and one extra zero in the second column.
\begin{proposition}\label{prop:st_sequence_permanent}
  For any $\mc{X}_{s, t} \ni x = (n_1, \ldots, n_t)$ and any $j \in \{1, 2, \ldots, t\}$, define
  \[T_j(x) = (n_1 + 1, \ldots, n_{j - 1} + 1, n_{j + 1}, \ldots, n_t, 0).\]
  Further, define $T_{t + 1}(x) = (n_1 + 1, \ldots, n_t + 1)$. If $x \not \in \mc{X}'_{s, t}$, then
  \begin{equation}
    \abs{A_{n, x}} = \sum_{j = 0}^t \abs{A_{n - 1, T_j(x)}}
  \end{equation}
  If $x \in \mc{X}'_{s, t}$, then
  \begin{equation}
    \abs{A_{n, x}} = \abs{A_{n - 1, T_1(x)}}
  \end{equation}
\end{proposition}

\begin{proof}
  The permanent of any matrix $A \in \mc{R}^{n \times n}$ is given by
  \[\abs{A} = \sum_{i = 1}^n A_{1i} \abs{A(1, i)},\]
  where for any $1 \le i, j \le n$, $A(i, j)$ denotes the matrix obtained by deleting the $i^{\text{th}}$ row and $j^{\text{th}}$ column from $A$.

  For any $1 \le i \le t + 1$, the matrix $A_{n, x}(1, i)$ is precisely $A_{n - 1, T_i(x)}$. As $A_{n, x}$ is a binary matrix,
  \[\abs{A_{n, x}} = \sum_{i = 1}^{t + 1} \abs{A_{n, x}(1, i)} = \sum_{j = 1}^{t + 1} \abs{A_{n - 1, T_j(x)}}.\]
  When $x \in \mc{X}'_{s, t}$, the first column of $A_{n, x}(1, i)$ is zero for all $i > 1$, and so
  \[\abs{A_{n, x}} = \abs{A_{n, x}(1, 1)} + \sum_{i = 2}^{t + 1} \abs{A_{n, x}(1, i)} = \abs{A_{n - 1, T_1(x)}}.\]
\end{proof}
In what follows, $x^*$ will be used to denote the state $(0, \ldots, 0)$. $A_{n, x^*}$ is simply the adjacency matrix of the type-$(s, t)$ graph, so $\abs{A_{n, x^*}} = \abs{\mc{M}_{n, s, t}}$. Additionally, since $A_{n, x^*}$ has entries at least as large as $A_{n, x}$ for any $x \in \mc{X}_{s, t}$, it follows that
\begin{equation}\label{eq:A_inequality}
  \abs{A_{n, x^*}} = \max_{x \in \mc{X}_{s, t}} \abs{A_{n, x}}
\end{equation}
The connection between these sequences and bipartite matchings is given in the following proposition.
\begin{proposition}\label{prop:sequence_matching_bijection}
  Let $\mc{M}_{n, x}$ be the set of matchings $\pi$ such that $A_{n, x}(i, \pi(i)) = 1$ for all $i$. Then, there exists a bijection between $\mc{M}_{n, x}$ and sequences $x_1, \ldots, x_n \in \mc{X}_{s, t}$ with the following properties:
  \begin{itemize}
    \item[a.]
      $x_1 = x$
    \item[b.]
      For all $i = 1, 2, \ldots, n - 1$, there exists $j_i \in \{1, 2, \ldots, \min(t + 1, n - i + 1)\}$ such that $x_{i + 1} = T_{j_i}(x_i)$
  \end{itemize}
\end{proposition}

As the sequence $x_1, \ldots, x_n$ is Markovian, imposing transition probabilities induces a distribution on type-$(s, t)$ matchings. In particular, with the time-dependent transition matrices
\begin{equation}\label{eq:st_transitions}
  K_i(x_i, T_j(x_i)) = \frac{\abs{A_{n - i, T_j(x_i)}}}{\abs{A_{n - i + 1, x_i}}},
\end{equation}
the resulting sequence is uniformly distributed on the space of type-$(s, t)$ sequences, resulting in an induced uniform distribution on matchings.

\begin{example}
  Suppose $n = 5$ and $s = t = 2$. The graph given by Figure~\ref{fig:22_graph}.
  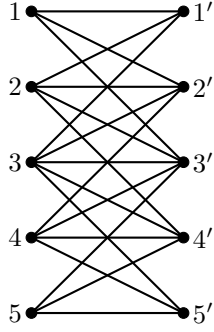
\begin{figure}
    \begin{center}
      \begin{tikzpicture}
        \draw[black, thick] (0, -2)--(2, -2);
        \draw[black, thick] (0, -2)--(2, -1);
        \draw[black, thick] (0, -2)--(2, -0);
        \draw[black, thick] (0, -1)--(2, -2);
        \draw[black, thick] (0, -1)--(2, -1);
        \draw[black, thick] (0, -1)--(2, 0);
        \draw[black, thick] (0, -1)--(2, 1);
        \draw[black, thick] (0, 0)--(2, -2);
        \draw[black, thick] (0, 0)--(2, -1);
        \draw[black, thick] (0, 0)--(2, 0);
        \draw[black, thick] (0, 0)--(2, 1);
        \draw[black, thick] (0, 0)--(2, 2);
        \draw[black, thick] (0, 1)--(2, -1);
        \draw[black, thick] (0, 1)--(2, 0);
        \draw[black, thick] (0, 1)--(2, 1);
        \draw[black, thick] (0, 1)--(2, 2);
        \draw[black, thick] (0, 2)--(2, 0);
        \draw[black, thick] (0, 2)--(2, 1);
        \draw[black, thick] (0, 2)--(2, 2);
        \filldraw[black] (0, 2) circle (2pt) node[anchor=east] {$1$};
        \filldraw[black] (2, 2) circle (2pt) node[anchor=west] {$1'$};
        \filldraw[black] (0, 1) circle (2pt) node[anchor=east] {$2$};
        \filldraw[black] (2, 1) circle (2pt) node[anchor=west] {$2'$};
        \filldraw[black] (0, 0) circle (2pt) node[anchor=east] {$3$};
        \filldraw[black] (2, 0) circle (2pt) node[anchor=west] {$3'$};
        \filldraw[black] (0, -1) circle (2pt) node[anchor=east] {$4$};
        \filldraw[black] (2, -1) circle (2pt) node[anchor=west] {$4'$};
        \filldraw[black] (0, -2) circle (2pt) node[anchor=east] {$5$};
        \filldraw[black] (2, -2) circle (2pt) node[anchor=west] {$5'$};
      \end{tikzpicture}
      \caption{The type-$(2, 2)$ graph with $n = 5$}\label{fig:22_graph}
    \end{center}
  \end{figure}
  The states in $\mc{X}_{s, t}$ are the pairs
  \[\{(0, 0), (1, 0), (1, 1), (2, 0), (2, 1), (2, 2)\}.\]
  Table~\ref{tab:22_permutation_sequence_map} shows the type-$(2, 2)$ sequences for several different type-$(2, 2)$ permutations.
  \begin{table}[h]
    \centering
    \begin{tabular}{c|c}
      $\pi$ & Type-$(2, 2)$ sequence \\ \hline
      $12345$ & $(0, 0), (0, 0), (0, 0), (0, 0), (0, 0)$ \\
      $23154$ & $(0, 0), (1, 0), (2, 0), (0, 0), (1, 0)$ \\
      $21435$ & $(0, 0), (1, 0), (0, 0), (1, 0), (0, 0)$ \\
      $31245$ & $(0, 0), (1, 1), (1, 0), (0, 0), (1, 0)$
    \end{tabular}
    \caption{Type-$(2, 2)$ sequences for various permutations}
    \label{tab:22_permutation_sequence_map}
  \end{table}
\end{example}

The remainder of the chapter will work with type-$(s, t)$ sequences instead of type-$(s, t)$ matchings. Therefore, it is helpful to rewrite Algorithm~\ref{alg:sequential} as an algorithm that samples elements of $\mc{X}^n_{s, t}$.
\begin{algorithm}\label{alg:sequential_sequence}
  Initialize $x_1 = x^*$. Given $x_1, \ldots, x_i$, for some $i \ge 1$:
  \begin{itemize}
    \item
      If $x_i = (n_{i1}, \ldots, n_{it})$ and $n_{i1} = s$, then set $x_{i + 1} = T_1(x_i)$ with probability 1. Otherwise, set $x_{i + 1} = T_I(x_i)$, where $I$ is uniformly chosen from $\{1, 2, \ldots, \min(t + 1, n - i + 1)\}$.
    \item
      This generates a random sequence $(x_1, \ldots, x_n) = X$ with probability
      \[\mu(X) = (t + 1)^{\theta(X) - n},\]
      where $\theta(X) = \abs{\{j: n_{j1} = s\}}$.
  \end{itemize}
\end{algorithm}

\subsection{Central limit theorem}
This section revisits Theorem~\ref{thm:TIMC_clt} and shows that the required conditions hold for the type-$(s, t)$ Markov chain. First, the variables $X_{n, 1}, \ldots, X_{n, n}$ are a realization of the Markov chain with transition matrices given by~\eqref{eq:st_transitions}, and so $X_{n, i} \in \mc{X}_{s, t}$ for all $n, i$. $Y_{n, i}$ is the indicator variable that $X_{n, i} \in \mc{X}'_{s, t}$, and so $C_n = 1$.

\begin{proposition}\label{prop:CLT_conditions}
  Let $\sigma_n$ and $\lambda_n$ be as defined in Theorem~\ref{thm:TIMC_clt}. Then,
  \begin{itemize}
    \item[1.]
      $\sigma_n \rightarrow \infty$ as $n \rightarrow \infty$.
    \item[2.]
      $\lambda_n \ge \epsilon > 0$ for some $\epsilon$ independent of $n$.
  \end{itemize}
\end{proposition}
With $C_n = 1$ and the results established by Proposition~\ref{prop:CLT_conditions}, the conditions~\eqref{eq:CLT_conditions_1} and~\eqref{eq:CLT_conditions_2} are satisfied, therefore proving the central limit theorem for $\theta(Y)$.
\begin{proof}[Proof of Proposition~\ref{prop:CLT_conditions}]
  The proof of part 1 relies on the following observations.
  \begin{itemize}
    \item[(a)]
      The diameter of the state space is at most $2t$. Indeed, for any states $x = (n_1, \ldots, n_t)$ and $y = (m_1, \ldots, m_t)$, $T_1^t(x) = x^*$ and
      \[T_{t + 1}^{t - m_t} T_1 T_{t + 1}^{m_t - m_{t - 1}} T_1 \cdots T_1 T_{t + 1}^{m_2 - m_1}(x^*) = y.\]
    \item[(b)]
      For all $i$, the nonzero entries of $K_i$ can be bounded away from zero:
      \[\frac{\abs{A_{n - i, T_j(x)}}}{\abs{A_{n - i + 1, x}}} \ge \frac{\abs{A_{n - i - t, x^*}}}{\abs{A_{n - i + 1, x^*}}} \ge \frac{1}{(t + 1)!}\]
  \end{itemize}
  \noindent $\sigma_n^2$ is the variance of the number of times the Markov chain visits $\mc{X}'_{s, t}$. Given states $x, y \in \mc{X}_{s, t}$, let $V(x, y, m)$ denote the number of visits to $\mc{X}'_{s, t}$ in a uniformly chosen sequence $(x_1, \ldots, x_m)$, conditioned on the event that $x_1 = x$ and $x_m = y$. Then, from the above observations, for $m = 10t$, $V(x, y, m)$ is a nonzero random variable with variance between $\delta_1(s, t)$ and $\delta_2(s, t)$, where $\delta_1(s, t) < \delta_2(s, t)$ are quantities that are independent of $n$.

  Fix states $x_{n, m}, x_{n, 2m}, \ldots, x_{n, m \cdot \floor{n/m}}$ and condition on the event $E$ that $X_{n, km} = x_{n, km}$ for $1 \le k \le \floor{n/m}$. Under this conditioning, the distribution of the states $Y_k = X_{n, km}, \ldots, X_{n, (k + 1)m}$ is the uniform distribution over type-$(s, t)$ sequences of length $m$ starting at $x_{n, km}$ and ending at $x_{n, (k + 1)m}$. Furthermore, $Y_1, \ldots, Y_{\floor{n/m} - 1}$ are conditionally independent due to the Markov property. Thus, the variance of $\sum_{i = 1}^n Y_{n, i}$ conditional on $E$ is
  \[(\floor{n/m} - 1) \delta_1(s, t) = \sum_{j = 1}^{\floor{n/m} - 1} \delta_1(s, t) \le \Var\left( \sum_{i = 1}^n Y_{n, i} \mid E\right) \le \sum_{j = 1}^{\floor{n/m} - 1} \delta_2(s, t) = (\floor{n/m} - 1) \delta_2(s, t).\]
  The Law of Total Variance therefore implies that
  \begin{equation}\label{eq:markov_variance_lower_bound}
    \Var\left( \sum_{i = 1}^n Y_{n, i} \right) \ge \E\left( \Var\left( \sum_{i = 1}^n Y_{n, i} \mid E \right) \right) = \Theta(n),
  \end{equation}
  meaning $\sigma_n = \Omega(\sqrt{n})$.

  For part 2, let $H_{s, t}$ denote the directed graph with vertex set $\mc{X}_{s, t}$ and an edge from $x$ to $y$ for all $y$ such that $y = T_i(x)$ for some $1 \le i \le t + 1$. Let $M_{s, t}$ denote its adjacency matrix.
  \begin{lemma}\label{lem:num_paths_ratio_convergence}
    Let $\lambda$ be the eigenvalue of the adjacency matrix of $H_{s, t}$ of maximum norm. Then,
    \begin{itemize}
      \item[1.]
        $\lambda$ is simple and real, and the corresponding right eigenvector $v$ can be chosen to have all positive coordinates.
      \item[2.]
        There exist $0 < \delta < 1$ and $N_\delta > 0$, both independent of $n$, such that for all $n > N_\delta$ and all $x, y \in \mc{X}_{s, t}$,
        \[\frac{\abs{A_{n, x}}}{\abs{A_{n, y}}} = \frac{v_x}{v_y} \left( 1 + O(e^{-n}) \right),\]
        where $v_x$ and $v_y$ are the $x$ and $y$ coordinates of $v$, respectively.
    \end{itemize}
  \end{lemma}

  \noindent The maximal correlation coefficient of $X_i$ and $X_{i + 1}$ can be expanded as follows:
  \begin{align}
    \rho(X_i, X_{i + 1}) &= \sup_{f, g} \E(f(X_i) g(X_{i + 1})) \nonumber \\
    &= \sup_{f, g} \E\left( \E(f(X_i) g(X_{i + 1}) \mid X_i) \right) \nonumber \\
    &= \sup_{f, g} \E\left( f(X_i) \E(g(X_{i + 1}) \mid X_i) \right) \nonumber \\
    &= \sup_{f, g} \E\left( f(X_i) \sum_{j = 1}^{t + 1} g(T_j(X_i)) K_{n, i}(x_i, T_j(x_i)) \right), \label{eq:max_correlation}
  \end{align}
  where the supremum is taken over all mean-zero, unit-variance functions $f$ and $g$.

  $\rho(X_i, X_{i + 1}) = 1$ if and only if there exist a pair of non-degenerate functions $f, g$ such that $f(X_i) = g(X_{i + 1})$ with probability 1. Such a pair exists if and only if the graph $H_{s, t}$ is bipartite; this is not the case, as $H_{s, t}$ has a self-loop at $x^*$.

  Secondly, observe that~\eqref{eq:max_correlation} gives that $\rho(X_i, X_{i + 1})$ is a convex function of the transition matrix $K_{n, i}$. Let $\mc{S}$ denote the set of permutations $\sigma \in S_{\abs{\mc{X}_{s, t}}}$ such that $K^*(i, \sigma(i)) > 0$ for all $i$. Next, define $\mc{C}_{s, t}$ to be the convex polytope in $\R^{\abs{\mc{X}_{s, t}}^2}$ with extreme points given by $\{P_\sigma\}_{\sigma \in \mc{S}}$, where $P_\sigma$ is the permutation matrix associated to $\sigma$.

  By Lemma~\ref{lem:num_paths_ratio_convergence}, there exists some $\epsilon, N_\epsilon > 0$, both independent of $n$, such that for all $n - i > N_\epsilon$, $K_{n, i} \in K^* \pm \epsilon \mc{C}_{s, t}$. Thus, $\rho(X_i, X_{i + 1})$ takes its maximum value at one of the vertices of $K^* \pm \epsilon \mc{C}_{s, t}$. Since the number of vertices is a function of $s, t$ and is independent of $n$, it therefore follows that $\rho(X_i, X_{i + 1})$ is bounded away from 1 for all $i$ satisfying $n - i > N_\epsilon$. Finally, the observation that the terms $Y_{n, j}$ for $n - j < N_\epsilon$ have negligible contribution to the left-hand side of~\eqref{eq:clt_statement} finishes the proof of the central limit theorem for $\theta(Y)$.
\end{proof}

\subsection{``Almost-perfect'' sampling}
\noindent This section constructs explicit sampling probabilities under which the log-density $\log \rho(Y)$ is a bounded random variable. In combination with the result of Chatterjee and Diaconis~\cite{cd18}, this gives an ``almost-perfect'' sampling algorithm; that is, only $O(1)$ samples are necessary and sufficient for importance sampling to converge.

A first idea is to sample using the probabilities
\[K_{n, i}(x_i, T_j(x_i)) = \frac{\abs{A_{n - i, T_j(x_i)}}}{\abs{A_{n - i + 1, x_i}}}.\]
Indeed, a matching sampled in this manner is \emph{exactly} uniformly distributed. However, computing these sampling probabilities requires evaluating the permanent of a large matrix, a computationally infeasible task. The goal of this section is to find easily computable probabilities that adequately approximate the uniform distribution.

Lemma~\ref{lem:num_paths_ratio_convergence} gives an indication of what the optimal sampling probabilities should be. Let $X_1, \ldots, X_n$ be a random sequence of elements in $\mc{X}_{s, t}$ satisfying
\begin{equation}\label{eq:st_opt_sampling}
  P(X_{i + 1} = T_j(x_i) \mid X_i = x_i) = K^*(x_i, T_j(x_i)) = \frac{v_{T_j(x_i)}}{\sum_{k = 1}^{t + 1} v_{T_k(x_i)}}
\end{equation}
Then, the sampling probability of the sequence $X_1, \ldots, X_n$ is given by
\[\tilde{\mu}(x_1, \ldots, x_n) = \I\{x_1 = x^*\} \prod_{i = 2}^n K^*(x_{i - 1}, x_i),\]
with density respect to the uniform given by
\[\frac{d\mu^*}{d\tilde{\mu}}(x_1, \ldots, x_n) = \prod_{i = 2}^n \frac{K_{n, i}(x_{i - 1}, x_i)}{K^*(x_{i - 1}, x_i)}.\]
By Lemma~\ref{lem:num_paths_ratio_convergence}, each factor in the product is of order $1 + O(e^{-n})$. Thus,
\[\frac{d\mu^*}{d\tilde{\mu}}(x_1, \ldots, x_n) = (1 + O(e^{-n}))^n = O(1).\]
\begin{example}[Fibonacci matchings]
  In the case $s = t = 1$, the state space $\mc{X}_{1, 1}$ consists of the two states $\{0, 1\}$, with the allowed transitions $0 \rightarrow 0, 0 \rightarrow 1$, and $1 \rightarrow 1$. The graph $H_{s, t}$ therefore has adjacency matrix
  \[M_{1, 1} = \begin{pmatrix}
    1 & 1 \\
    1 & 0
  \end{pmatrix}\]
  $M_{1, 1}$ has right Perron-Frobenius eigenvector $v^T = (\varphi, 1)$, where $\varphi = \frac{1 + \sqrt{5}}{2}$. By~\eqref{eq:st_opt_sampling}, the ``almost-perfect'' sampling probabilities are
  \begin{align*}
    P(X_{i + 1} = 0 \mid X_i = 0) &= \frac{1}{\varphi} \\
    P(X_{i + 1} = 1 \mid X_i = 0) &= \frac{1}{\varphi^2} \\
    P(X_{i + 1} = 0 \mid X_i = 1) &= 1.
  \end{align*}
  When sampling matchings, this yields
  \begin{align*}
    P(\pi_i = i) &= \frac{1}{\varphi} \\
    P(\pi_i = i + 1) &= \begin{cases}
      \frac{1}{\varphi^2} &\mbox{if~} (i - 1)' \text{~has been matched} \\
      0 &\mbox{otherwise}
    \end{cases}
  \end{align*}
\end{example}
\begin{example}[Distance-2 matchings]
  In the case $s = t = 2$, the state space $\mc{X}_{2, 2}$ consists of the six states
  \[\mc{X}_{2, 2} = \{(0, 0), (1, 0), (1, 1), (2, 0), (2, 1), (2, 2)\}.\]
  The state space graph $H_{2, 2}$ has the adjacency matrix
  \[M_{2, 2} = \begin{pmatrix}
    1 & 1 & 1 & 0 & 0 & 0 \\
    1 & 0 & 0 & 1 & 1 & 0 \\
    0 & 1 & 0 & 1 & 0 & 1 \\
    1 & 0 & 0 & 0 & 0 & 0 \\
    0 & 1 & 0 & 0 & 0 & 0 \\
    0 & 0 & 0 & 1 & 0 & 0
  \end{pmatrix}\]
  The Perron-Frobenius eigenvalue $\lambda$ solves the polynomial equation $\lambda^5 - 2 \lambda^4 - 2\lambda^2 + 1 = 0$, and the associated eigenvector is given by
  \[v^T = \left(\lambda^2, \frac{\lambda^2}{\lambda - 1}, \lambda^3 \cdot \frac{\lambda - 2}{\lambda - 1}, \lambda, \frac{\lambda}{\lambda - 1}, 1 \right).\]
  The optimal sampling probabilities can therefore be described as follows:
  \begin{itemize}
    \item[1.]
      If the vertices $[i - 1]$ and $[(i - 1)']$ have all been matched, then
      \[\pi_i = \begin{cases}
        i &\mbox{w.p.~} \frac{\lambda - 1}{\lambda^2} \\
        i + 1 &\mbox{w.p.~} \frac{1}{\lambda^2} \\
        i + 2 &\mbox{w.p.~} 1 - \frac{1}{\lambda}
      \end{cases}\]
    \item[2.]
      If the vertices $[i - 1]$ are matched to $\{1', \ldots, (i - 2)', i'\}$, then
      \[\pi_i = \begin{cases}
        i - 1 &\mbox{w.p.~} 1 - \frac{1}{\lambda} \\
        i + 1 &\mbox{w.p.~} \frac{\lambda - 1}{\lambda^2} \\
        i + 2 &\mbox{w.p.~} \frac{1}{\lambda^2}
      \end{cases}\]
    \item[3.]
      If the vertices $[i - 1]$ are matched to $\{1', \ldots, (i - 2)', (i + 1)'\}$, then
      \[\pi_i = \begin{cases}
        i - 1 &\mbox{w.p.~} \frac{\lambda^2}{2\lambda^2 - 1} \\
        i &\mbox{w.p.~} \frac{\lambda(\lambda - 1)}{2 \lambda^2 - 1} \\
        i + 2 &\mbox{w.p.~} \frac{\lambda - 1}{2\lambda^2 - 1}
      \end{cases}\]
  \end{itemize}
\end{example}
Both of these examples match the optimal sampling probabilities found by Diaconis and Kolesnik~\cite{dk19}.

\section{Analysis of Moments}\label{sec:moments}
\noindent This section focuses on the asymptotic behavior of $\E \log \rho(Y)$ and $\Var \log \rho(Y)$ under Algorithm~\ref{alg:sequential}. Although the previous section derives an improvement to the simple sequential algorithm, computing the constants in the exponent of the required sample size is still a worthwile endeavor. This is because the almost-perfect sampling is specific to type-$(s, t)$ graphs and therefore is not applicable to more general bipartite graphs.

By Proposition~\ref{prop:st_sampling_formula},
\begin{align}
  \log \rho(Y) &= (n - \theta(Y)) \log (t + 1) - \log \abs{A_{n, x^*}} + O_{s, t}(1) \nonumber \\
  \E \log \rho(Y) &= (n - \E \theta(Y)) \log(t + 1) - \log \abs{A_{n, x^*}} + O_{s, t}(1) \label{eq:rho_expectation} \\
  \Var \log \rho(Y) &= \log^2(t + 1) \Var \theta(Y) + O(\sqrt{\Var \theta(Y)}). \label{eq:rho_variance}
\end{align}
It suffices to analyze the asymptotics of $\theta(Y)$, since it is the only source of randomness in $\log \rho(Y)$. In particular, it will be shown that both $\frac{1}{n} \E \log \rho(Y)$ and $\frac{1}{n} \Var \log \rho(Y)$ converge to constants $e_{s, t}$ and $v_{s, t}$ which depend on $s$ and $t$. Further, analysis of the generating function
\begin{equation}\label{eq:st_generating_function}
  G_{s, t}(y, z) = \sum_{n = 0}^\infty z^n \sum_{\pi \in \mc{M}_{n, s, t}} y^{\theta(\pi)}
\end{equation}
yields the exact values for $e_{s, t}$ and $v_{s, t}$, for the pairs $(s, t)$ given in Table~\ref{tab:rho_asymptotics}.

The first step is to show that both $\E \log \rho(Y)$ and $\Var \log \rho(Y)$ grow linearly in $n$. To this end, recall from the proof of Proposition~\ref{prop:CLT_conditions} that the diameter $d$ of the state space is bounded above by $2t$. Thus, there exists $\epsilon = \epsilon(s, t) > 0$ such that the matrices
\begin{equation}\label{eq:diameter}
  \tilde{K}_i = \prod_{j = i}^{i + 2t} K_i
\end{equation}
have entries in $[\epsilon(s, t), 1 - \epsilon(s, t)]$.
\[\epsilon(s, t) \cdot \frac{n}{2t} \le \E \theta(Y) \le (1 - \epsilon(s, t)) \cdot \frac{2t - 1}{2t} n,\]
showing that $\E \log \theta(Y) = \Theta(n)$.

Next, note that an $O(n)$ lower bound for the variance is established by~\eqref{eq:markov_variance_lower_bound}. It remains to derive an upper bound for $\Var \theta(Y)$. To this end, let $X = (X_1, \ldots, X_n)$ be a uniformly random element of $\mc{T}_{n, s, t}$. For each $0 \le i \le n - 1$, generate a sequence $X^i_{i + 1}, \ldots, X^i_n$ from the conditional distribution of $(X_{i + 1}, \ldots, X_n)$ given $(X_1, \ldots, X_i)$, but conditionally independent of $(X_{i + 1}, \ldots, X_n)$. This is done by resampling the chain using Algorithm~\ref{alg:sequential_sequence} starting from index $i$.

Let $\tau_i := \min\{j \ge i + 1: X_j = X_j^i\}$, with $\tau_i := n + 1$ if $X_j \ne X_j^i$ for all $i + 1 \le j \le n$. Define $Y_j^i := X_j^i$ for $i + 1 \le j < \tau_i$ and $Y_j^i = X_j$ for $j \ge \tau_i$. The following lemma asserts that the conditional distributions of $\{X_j^i\}$ and $\{Y_j^i\}$ are the same.
\begin{lemma}\label{lem:conditional_dist}
  For $0 \le i \le n - 2$, the conditional distributions of $(Y_{i + 2}^i, \ldots, Y_n^i)$ and $(X_{i + 2}^i, \ldots, X_n^i)$ given $(X_1, \ldots, X_{i + 1}, X_{i + 1}^i)$ are the same. Also, for each $0 \le i \le n - 1$, $Y_{i + 1}^i = X_{i + 1}^i$.
\end{lemma}
Now, starting from $X$, define the random vector $Y$ by first choosing an index $I$ uniformly at random from $\{0, \ldots, n - 1\}$, and then defining
\[Y_j = \begin{cases}
  X_j &\mbox{if~} j \le i \\
  Y_j^i &\mbox{otherwise}
\end{cases}\]
Lemma~\ref{lem:conditional_dist} then asserts that $X$ and $Y$ have the same distribution. Furthermore, the martingale decomposition of variance can be used to achieve the following bound on $f(X)$ for any initial state $x$ and any function $f: \mc{T}_{n, s, t}(x) \rightarrow \R$.
\begin{lemma}\label{lem:var_bound}
  For any $f: \mc{T}_{n, s, t}(x) \rightarrow \R$,
  \[\Var(f(X)) \le \frac{n}{2} \E(f(X) - f(Y))^2.\]
\end{lemma}
Taking $f(X) = \theta(X)$ means that $f(X) - f(Y)$ is bounded above by $\tau_I - I$, where $I$ is the random index used to construct $Y$. To derive the asymptotic upper bound on $\Var f(X)$, it therefore suffices to show that $\E \tau_I^2 = O(1)$.

From~\eqref{eq:diameter}, the $d$-step transitions $\prod_{j = 0}^d K_{i + j}$ have entries that are all bounded below by $\epsilon = \epsilon(s, t) > 0$. For each $k$, let
\[\tilde{\tau}_k = \min\{j > 0: X_{k + jd} = X^k_{k + jd}\}.\]
Further, let $Z_\epsilon$ be a geometric random variable with parameter $\epsilon$. Then, $\tau_k - k \le d \cdot \tilde{\tau}_k$ almost surely, and $\E \tilde{\tau}_k^2 \le \E Z_\epsilon^2$. Thus,
\[\E \tau_k^2 \le \frac{d^2}{\epsilon^2} \le \frac{4t^2}{\epsilon^2}\]
Putting this together with Lemma~\ref{lem:var_bound} gives $\Var \theta(Y) \le \frac{2t^2}{\epsilon^2} n$.
\subsection{Generating functions}\label{sec:generating_functions}
Thus far, it has been shown that both $\E \theta(Y)$ and $\Var \theta(Y)$ are both of order $n$. This section explicitly computes the asymptotic behavior of $\E \theta(Y)$ and $\Var \theta(Y)$ using generating functions. This method was first analyzed in~\cite{cdg18} and later refined in~\cite{dk19} to be applicable to type-$(1, 1)$, type-$(2, 1)$, and type-$(2, 2)$ permutations. This section further generalizes the method to arbitrary pairs $(s, t)$.

For each $x \in \mc{X}_{s, t}$, let $B_{n, x, y}$ be the matrix with entries in $\{0, 1, y\}$ satisfying the following conditions:
\begin{itemize}
  \item
    If $A_{n, x}(i, j) = 0$, then $B_{n, x, y}(i, j) = 0$.
  \item
    If $A_{n, x}(i, j) = 1$ and $A_{n, x}(i - 1, j) = A_{n, x}(i, j - 1) = 0$, then $B_{n, x, y}(i, j) = y$. Here, $A_{n, x}(i, j)$ is assumed to be 0 if either $i$ or $j$ is negative.
  \item
    For all other pairs $(i, j)$, $B_{n, x, y}(i, j) = 1$.
\end{itemize}
For example, for $n = 8, s = 3, t = 2$, and $x = (-2, 0)$,
\[B_{n, x, y} = \begin{pmatrix}
  1 & 1 & 1 & 0 & 0 & 0 & 0 & 0 \\
  y & 1 & 1 & 1 & 0 & 0 & 0 & 0 \\
  0 & 1 & 1 & 1 & 1 & 0 & 0 & 0 \\
  0 & y & 1 & 1 & 1 & 1 & 0 & 0 \\
  0 & 0 & 1 & 1 & 1 & 1 & 1 & 0 \\
  0 & 0 & y & 1 & 1 & 1 & 1 & 1 \\
  0 & 0 & 0 & y & 1 & 1 & 1 & 1 \\
  0 & 0 & 0 & 0 & y & 1 & 1 & 1
\end{pmatrix}.\]
Notice that $B_{n, x, 1} = A_{n, x}$, and that
\[\abs{B_{n, x^*, y}} = \sum_{\pi \in \mc{M}_{n, s, t}} y^{\theta(\pi)}.\]
Next, define the generating function
\begin{equation}\label{eq:theta_generating_function}
  G_{s, t}(y, z) = \sum_{n = 0}^\infty z^n \sum_{\pi \in \mc{M}_{n, s, t}} y^{\theta(\pi)} = \sum_{n = 0}^{\infty} z^n \abs{B_{n, x^*, y}}.
\end{equation}
Under the uniform distribution, the $k^{\text{th}}$ falling moment of $\theta(\pi)$ is given by
\[\abs{\mc{M}_{n, s, t}} \E\left[ \theta(\pi) (\theta(\pi) - 1) \cdots (\theta(\pi) - k + 1) \right] = [z^n]\left( \frac{\partial^k}{\partial y^k} G_{s, t}(y, z) \biggr|_{y = 1} \right),\]
where $[z^n] f(z)$ denotes the coefficient of $z^n$ in the power series expansion of $f$.

When $f$ is expressible as a rational function $\frac{P(z)}{Q(z)}$, where $Q(z)$ has roots $r_1, r_2, \ldots, r_m$, then the partial fraction decomposition of $f$ is given by
\[\frac{P(z)}{Q(z)} = \sum_{i = 1}^m \left( \frac{a_{i1}}{1 - \frac{z}{r_i}} + \frac{a_{i2}}{\left( 1 - \frac{z}{r_i} \right)^2} + \cdots + \frac{a_{i k_i}}{\left( 1 - \frac{z}{r_i} \right)^{k_i}} \right),\]
The coefficients $a_{ij}$ are computable using the residue method and are given by
\begin{equation}\label{eq:residue}
  a_{ij} = \frac{1}{(-r_i)^j (k_i - j)!} \lim_{z \rightarrow r_i} \frac{d^{k_i - j}}{d z^{k_i - j}} \left( (z - r_i)^{k_i} \frac{P(z)}{Q(z)} \right).
\end{equation}
When $n$ is large, the main contributions to the coefficient of $z^n$ in $f(z)$ are from the terms corresponding to the root with the smallest magnitude.

The following examples explicitly compute the generating functions for various pairs $(s, t)$.
\begin{example}[$t = 1$]\label{ex:s1_generating_func}
  When $t = 1$, the state space $\mc{X}_{s, t}$ is the set of integers $\{-s, -s + 1, \ldots, 0\}$, and the subset $\mc{X}'_{s, t}$ is the singleton state $\{-s\}$. The transitions between states are given by $T_1(x) = 0 = x^*$ and $T_2(x) = x - 1$. The permanents $\abs{B_{n, x, y}}$ satisfy the recursion
  \begin{equation}
    \abs{B_{n, x, y}} = \begin{cases}
      \abs{B_{n - 1, T_1(x), y}} + \abs{B_{n - 1, T_2(x), y}} &\mbox{if~} x \not\in \mc{X}'_{s, t} \\
      \abs{B_{n - 1, T_1(x), y}}y &\mbox{otherwise}
    \end{cases}
  \end{equation}
  Further expansion yields
  \begin{align}
    \abs{B_{n, x^*, y}} = \begin{cases}
      \sum_{i = 1}^s \abs{B_{n - i, x^*, y}} + \abs{B_{n - s - 1, x^*, y}}y &\mbox{if~} n \ge s + 1 \\
      \sum_{i = 1}^n 2^{i - 1} y^i &\mbox{otherwise}
    \end{cases}
  \end{align}
  The generating function $G_{s, 1}(y, z)$ can then be written as the rational function
  \begin{equation}\label{eq:s1_generating_function}
    G_{s, 1}(y, z) = \frac{z + z^2 + \cdots + z^s + yz^{s + 1}}{1 - z - z^2 - \cdots - z^s - yz^{s + 1}},
  \end{equation}
  with derivatives
  \begin{align}
    \left( \frac{d}{dy} G_{s, 1}(y, z) \right)_{y = 1} &= \frac{z^{s + 1}}{(1 - z - z^2 - \cdots - z^{s + 1})^2} \\
    \left( \frac{d^2}{dy^2} G_{s, 1}(y, z) \right)_{y = 1} &= \frac{2z^{2s + 2}}{(1 - z - z^2 - \cdots - z^{s + 1})^3}
  \end{align}
  Observe that the polynomial $1 - z - z^2 - \cdots - z^{s + 1}$ has one simple root $r$ on the positive real line, and that $r$ is the root of lowest magnitude. It therefore follows that there exist constants $c_1, \ldots, c_7$ such that for large $n$,
  \begin{align*}
    \abs{M_{n, s, 1}} &= \frac{c_1}{r^n} + o(1) \\
    \abs{M_{n, s, 1}} \E \theta(\pi) &= \frac{c_2 + c_3 n}{r^n} + o(1) \\
    \abs{M_{n, s, 1}} \E(\theta(\pi)(\theta(\pi) - 1)) &= \frac{c_4 + c_5 n + c_6 n^2}{r^n} + o(1).
  \end{align*}
  This means that the expectation and variance of $\theta(\pi)$ are
  \begin{align}
    \E \theta(\pi) &= \frac{c_2 + c_3 n}{c_1} \label{eq:theta_mean} + o(n) \\
    \Var \theta(\pi) &= \frac{c_2 + c_4 + (c_3 + c_5)n + c_6 n^2}{c_1} - \left( \frac{c_2 + c_3 n}{c_1} \right)^2 + o(n). \label{eq:theta_var}
  \end{align}
  By Lemma~\ref{lem:var_bound}, the coefficient of $n^2$ in~\eqref{eq:theta_var} is necessarily zero, and indeed, numerical computations confirm this fact. The full results of this computation are listed in Table~\ref{tab:theta_asymptotics} for various values of $s$.
\end{example}

\begin{example}[$s = 3, t = 2$]\label{ex:32_generating_func}
  The state space is
  \[\mc{X}_{s, t} = \{(0, 1), (-1, 1), (-1, 0), (-2, 1), (-2, 0), (-2, -1), (-3, 1), (-3, 0), (-3, -1), (-3, -2)\},\]
  with $\mc{X}'_{s, t}$ being comprised of the last four states. In addition to the state $x^* = (0, 1)$, let $x' = (-1, 1)$. After simplification, the quantities $\abs{B_{n, x, y}}$ can be shown to satisfy the recursive relations
  \begin{align}
    \abs{B_{n, x^*, y}} &= \abs{B_{n - 1, x^*, y}} + \abs{B_{n - 3, x^*, y}} + (1 + 2y) \abs{B_{n - 4, x^*, y}} + (y + y^2) \abs{B_{n - 5, x^*, y}} \nonumber \\
    &\qquad+ \abs{B_{n - 1, x', y}} + \abs{B_{n - 2, x', y}} + y \abs{B_{n - 4, x', y}} + y \abs{B_{n - 5, x', y}} \label{eq:32_B1} \\
    \abs{B_{n, x', y}} &= \abs{B_{n - 1, x^*, y}} + \abs{B_{n - 2, x^*, y}} + 2y \abs{B_{n - 3, x^*, y}} + y \abs{B_{n - 4, x^*, y}} + y^2 \abs{B_{n - 5, x^*, y}} \nonumber \\
    &\qquad+ \abs{B_{n - 2, x', y}} + y \abs{B_{n - 3, x', y}} + y^2 \abs{B_{n - 5, x', y}}. \label{eq:32_B2}
  \end{align}
  Next, define the auxiliary generating function
  \begin{equation}\label{eq:aux_generating_function}
    H_{3, 2}(y, z) = \sum_{n = 0}^\infty z^n \abs{B_{n, x', y}}.
  \end{equation}
  With $G^{(i)}_{3, 2}$ and $H^{(i)}_{3, 2}$ denoting the partial sums of $G_{3, 2}$ and $H_{3, 2}$,~\eqref{eq:32_B1} and~\eqref{eq:32_B2} imply that
  \begin{align}
    f_1(y, z) &= g_1(y, z) G_{3, 2}(y, z) + h_1(y, z) H_{3, 2}(y, z) \label{eq:32_generating1} \\
    f_2(y, z) &= g_2(y, z) G_{3, 2}(y, z) + h_2(y, z) H_{3, 2}(y, z), \label{eq:32_generating2}
  \end{align}
  where
  \begin{align*}
    f_1(y, z) &= z G^{(1)}_{3, 2} + z^3 G^{(3)}_{3, 2} + z^4 (1 + 2y) G^{(4)}_{3, 2} + z^5(y + y^2) G^{(5)}_{3, 2} + z H^{(1)}_{3, 2} + z^2 H^{(2)}_{3, 2} + z^4 y H^{(4)}_{3, 2} + z^5 y H^{(5)}_{3, 2} \\
    g_1(y, z) &= z^5(y + y^2) + z^4(1 + 2y) + z^3 + z - 1 \\
    h_1(y, z) &= z + z^2 + z^4 y + z^5 y \\
    f_2(y, z) &= z G^{(1)}_{3, 2} + z^2 G^{(2)}_{3, 2} + 2z^3 y G^{(3)}_{3, 2} + z^4 y G^{(4)}_{3, 2} + z^5 y^2 G^{(5)}_{3, 2} + z^2 H^{(2)}_{3, 2} + z^3 y H^{(3)}_{3, 2} + z^5 y^2 H^{(5)}_{3, 2} \\
    g_2(y, z) &= z + z^2 + 2z^3 y + z^4 y + z^5 y^2 \\
    h_2(y, z) &= z^5 y^2 + z^3 y + z^2 - 1,
  \end{align*}
  and the partial sums $G^{(i)}_{3, 2}$ and $H^{(i)}_{3, 2}$ are given by
  \begin{align*}
    G^{(1)}_{3, 2} &= z \\
    G^{(2)}_{3, 2} &= z + 2z^2 \\
    G^{(3)}_{3, 2} &= z + 2z^2 + 6z^3 \\
    G^{(4)}_{3, 2} &= z + 2z^2 + 6z^3 + (4y + 12)z^4 \\
    G^{(5)}_{3, 2} &= z + 2z^2 + 6z^3 + (4y + 12)z^4 + (y^2 + 14y + 27) z^5 \\
    H^{(1)}_{3, 2} &= z \\
    H^{(2)}_{3, 2} &= z + 2z^2 \\
    H^{(3)}_{3, 2} &= z + 2z^2 + (2y + 4)z^3 \\
    H^{(4)}_{3, 2} &= z + 2z^2 + (2y + 4)z^3 + (4y + 10)z^4 \\
    H^{(5)}_{3, 2} &= z + 2z^2 + (2y + 4)z^3 + (4y + 10)z^4 + (y^2 + 13y + 24)z^5
  \end{align*}
  The solution to~\eqref{eq:32_generating1} and~\eqref{eq:32_generating2} is then given by
  \begin{equation}\label{eq:32_generating_solution}
    \begin{pmatrix} G \\ H
    \end{pmatrix} = \frac{1}{g_1 h_2 - h_1 g_2} \begin{pmatrix}
      h_2 & -h_1 \\ -g_2 & g_1
    \end{pmatrix} \cdot \begin{pmatrix} f_1 \\ f_2
    \end{pmatrix}
  \end{equation}
  As with Example~\ref{ex:s1_generating_func}, the asymptotic behavior for the expectation and variance of $\theta(\pi)$ is computed and displayed in Table~\ref{tab:theta_asymptotics}, albeit with aid from Maple.
\end{example}

\begin{table}[h]
  \centering
  \begin{tabular}{c|c|c|c}
    $(s, t)$ & $\frac{\E \theta(\pi)}{n}$ & $\frac{\Var \theta(\pi)}{n}$ \\ \hline
    $(2, 1)$ & 0.09939 & 0.05950 \\
    $(3, 1)$ & 0.04102 & 0.03138 \\
    $(4, 1)$ & 0.01832 & 0.01580 \\
    $(5, 1)$ & 0.00857 & 0.00788 \\
    $(6, 1)$ & 0.00412 & 0.00393 \\
    $(7, 1)$ & 0.00201 & 0.00196 \\
    $(3, 2)$ & 0.09077 & 0.06061
  \end{tabular}
  \caption{Asymptotics for $\theta(\pi)$}
  \label{tab:theta_asymptotics}
\end{table}
Consequently, the mean and variance of $\log \rho(Y)$, given by~\eqref{eq:rho_expectation} and~\eqref{eq:rho_variance}, are given in Table~\ref{tab:rho_asymptotics}.
\begin{table}[h]
  \centering
  \begin{tabular}{c|c|c|c}
    $(s, t)$ & $\frac{\E \log \rho(Y)}{n}$ & $\frac{\Var \log \rho(Y)}{n}$ \\ \hline
    $(2, 1)$ & 0.01488 & 0.02859 \\
    $(3, 1)$ & 0.00846 & 0.01508 \\
    $(4, 1)$ & 0.00448 & 0.00759 \\
    $(5, 1)$ & 0.00230 & 0.00379 \\
    $(6, 1)$ & 0.00117 & 0.00189 \\
    $(7, 1)$ & 0.00059 & 0.00094 \\
    $(3, 2)$ & 0.04041 & 0.07315
  \end{tabular}
  \caption{Asymptotics for $\log \rho(Y)$}
  \label{tab:rho_asymptotics}
\end{table}

By the result of Chatterjee and Diaconis~\cite{cd18}, the necessary and sufficient sample size for importance sampling is roughly $\exp(\E \log \rho(Y) + \sqrt{\Var(\log \rho(Y))})$. Combining this with the $O(n)$ time required for Algorithm~\ref{alg:sequential} to produce a sample, yields an aggregate runtime of
\begin{equation}\label{eq:sequential_runtime}
  O(n \exp(\E \log \rho(Y) + \sqrt{\Var \log \rho(Y)}))
\end{equation}
for approximating the uniform distribution. Table~\ref{tab:sample_size_asymptotics} reports this sample size for various pairs $(s, t)$ and graph sizes $n$.

\begin{table}
  \centering
  \begin{tabular}{c|c|c|c|c}
    $(s, t)$ & $n = 100$ & $n = 200$ & $n = 500$ & $n = 1000$ \\ \hline
    $(2, 1)$ & 25 & 215 & 74,607 & 607,766,194 \\
    $(3, 1)$ & 8 & 31 & 1068 & 228,228 \\
    $(4, 1)$ & 4 & 9 & 66 & 1382 \\
    $(5, 1)$ & 3 & 4 & 13 & 71 \\
    $(6, 1)$ & 2 & 3 & 5 & 13 \\
    $(7, 1)$ & 2 & 2 & 3 & 5 \\
    $(3, 2)$ & 851 & 148,372 & $2.52 \cdot 10^{11}$ & $1.84 \cdot 10^{21}$
  \end{tabular}
  \caption{Sample size requirements for convergence of importance sampling}
  \label{tab:sample_size_asymptotics}
\end{table}

It is worth pausing here to compare the results of the sequential importance sampling algorithm with the best bounds for the switch chain algorithm. After all, the importance sampling algorithm is asymptotically exponential, while the switch chain has been shown to mix in polynomial time. However, as the constants in Table~\ref{tab:rho_asymptotics} are quite small, it turns out that for small to moderate values of $n$, the importance sampling algorithm outperforms its Markov chain counterpart.

While the central limit theorem for $\log \rho(Y)$ holds for arbitrary pairs $(s, t)$, computing the constants as in Table~\ref{tab:rho_asymptotics} using generating functions is a challenging task for large pairs. For example, when $s = t = 6$, the generating function $G_{6, 6}(1, z)$ is a rational function $\frac{P(z)}{Q(z)}$, where $\deg(P(z)) = 482$ and $\deg(Q(z)) = 494$ (see oeis.org, entry a002524).

\section{Conclusions and future work}\label{sec:future}
\noindent The results presented in this paper demonstrate that importance sampling is an attractive alternative to the MCMC algorithms in the computer science literature for sampling matchings from type-$(s, t)$ graphs. While current techniques are promising for small $s, t$, they quickly become algebraically and computationally intensive for more complex cases. One future direction of work is to find a more tractable way of computing the asymptotic moments of $\log \rho(Y)$.

While importance sampling is practical for type-$(s, t)$ graphs, little is known about its performance for other classes of bipartite graphs. In particular, switch Markov chain proposed by Diaconis et.\ al.~\cite{dgh01} is applicable for the larger class of monotone graphs. For those graphs, Algorithm~\ref{alg:sequential} is less efficient, as checking whether or not a partial matching can be completed to a perfect matching is a more involved process. It will be interesting to see if an efficient importance sampling algorithm exists for monotone graphs, and if the techniques in this paper apply in the more general setting.

The almost-perfect sampling probabilities derived in this chapter rely on the bijection between matchings of type-$(s, t)$ graphs and Markovian sequences. Such a mapping does not necessarily exist for more general bipartite graphs. In the general case, it has been shown empirically that Sinkhorn balancing the adjacency matrix of the graph can yield sampling probabilities that outperform the uniform sampling in Algorithm~\ref{alg:sequential}. In particular, Beichl and Sullivan~\cite{bs99} demonstrate this in the context of counting dimer coverings of a lattice. Quantifying this improvement for different classes of bipartite graphs is a worthwile research problem. Curiously, while Sinkhorn balancing appears to improve the performance of the sequential algorithm on type-$(s, t)$ graphs, it does \emph{not} give the optimal sampling probabilities derived in Section~\ref{sec:markov}.

\section{Appendix}\label{sec:proofs}
\subsection{Proof of Proposition~\ref{prop:alg_success}}
Let $\mc{M}_{i - 1}$ be the partial matching of the vertices $1, 2, \ldots, i - 1$. The vertex $(i - s)'$ is connected to the vertices $(i - s - t)_+, \ldots, i$, and so if $(i - s)'$ is not matched to any of the vertices $1, 2, \ldots, i - 1$, then any perfect matching of $G$ containing $\mc{M}_{i - 1}$ must match $i$ with $(i - s)'$. This means that $J_i \subseteq \{(i - s)'\}$.

Conversely, if after step $i$, the unmatched vertices in $[n']$ are $\{v_1' < \cdots < v_{n - i}'\}$, where $v_1 > i - s$, then the perfect matching is completable by matching $i + k$ to $v_k'$ for all $1 \le k \le n - i$.
\subsection{Proof of Proposition~\ref{prop:st_sampling_formula}}
Let $\mc{M}_{i - 1}$ be the partial matching of the vertices $1, 2, \ldots, i - 1$. The vertex $(i - s)'$ is connected to the vertices $(i - s - t)_+, \ldots, i$, and so if $(i - s)'$ is not matched to any of the vertices $1, 2, \ldots, i - 1$, then any perfect matching of $G$ containing $\mc{M}_{i - 1}$ must match $i$ with $(i - s)'$. This means that $J_i \subseteq \{(i - s)'\}$.

Conversely, if after step $i$, the unmatched vertices in $[n']$ are $\{v_1' < \cdots < v_{n - i}'\}$, where $v_1 > i - s$, then the perfect matching is completable by matching $i + k$ to $v_k'$ for all $1 \le k \le n - i$.
\subsection{Proof of Proposition~\ref{prop:sequence_matching_bijection}}
First, note that for any $x \in \mc{X}_{s, t}$, the matrix $A_{k - 1, T_j(x)}$ is the result of removing the first row and $j^{\text{th}}$ column from $A_{k, x}$.

Write $\pi = (\pi_1, \pi_2, \ldots, \pi_n)$. Then, $A_{n, x^*}(i, \pi_i) = 1$ for all $1 \le i \le n$. Define $x_1, \ldots, x_n$ so that for all $i = 1, \ldots, n$, $A_{n - i + 1, x_i}$ is the matrix formed by deleting rows $1, 2, \ldots, i - 1$ and columns $\pi_1, \ldots, \pi_{i - 1}$ from $A_{n, x^*}$. It then is immediately clear that $x_1 = x^*$, and $x_{k + 1} = T_{j_k}(x_k)$ for some $j_1, \ldots, j_{n - 1} \in \{1, 2, \ldots, t + 1\}$. Additionally, there are only $n - i + 1$ columns in $A_{n - i + 1, x_i}$, so $j_i \le n - i + 1$. This shows that $x_1, \ldots, x_n$ is a sequence satisfying conditions a.\ and b.\ of the proposition.

Conversely, given $x_1, \ldots, x_n$ satisfying the two conditions, construct the matching $\pi$ by the following procedure.
\begin{itemize}
  \item[1.]
    Initialize $\pi = \{\}$ and $\sigma = \{1, 2, \ldots, n\}$.
  \item[2.]
    For $i = 1, 2, \ldots, n - 1$, remove the $j_i^{\text{th}}$ smallest element from $\sigma$ and add it to $\pi$. This step is always possible because $j_i \le n - i + 1$.
  \item[3.]
    This leaves one element left in $\sigma$. Add that element to the end of $\pi$.
\end{itemize}
At each step, the $(t + 1)^{\text{st}}$ smallest element in $\sigma$ is at most $i + t$. Thus, since $j_i \le t + 1$, it must be that $\pi_i \le i + t$.

Next, an inductive argument shows that at step $i$, if $x_i = (n_1, \ldots, n_t)$, then the smallest $t$ elements in $\sigma$ are $i - n_1, i + 1 - n_2, \ldots, i + t - 1 - n_t$. Since $n_1 \le s$, this immediately shows that $\pi_i \ge i - s$. Thus, $\pi$ is a perfect matching of the type-$(s, t)$ graph.
\subsection{Proof of Lemma~\ref{lem:num_paths_ratio_convergence}}
The first claim is proved using the Perron-Frobenius theorem. It suffices to show that $H_{s, t}$ is a strongly connected graph with period 1.

That $H_{s, t}$ is aperiodic follows from the fact that it has a self-loop at $x_0$. Next, for any two states $x, y \in \mc{X}_{s, t}$, a directed path exists from $x$ to $y = (n_1, \ldots, n_t)$ through the state $x_0 = (0, 1, 2, \ldots, t)$:
\begin{align*}
  x_0 &= T_1^{t}(x) \\
  y &= T_{t + 1}^{k_{t - 1}} \circ T_1 \circ T_{t + 1}^{k_{t - 2}} \circ T_1 \circ \cdots \circ T_{t + 1}^{k_2} \circ T_1 \circ T_{t + 1}^{k_1} (x_0),
\end{align*}
where $k_i = y_{i + 1} - y_i - 1$. This shows that $H_{s, t}$ is strongly connected.

For the second claim, let $P_n(x, y)$ denote the collection of directed paths of length $n$ that go from $x$ to $y$. Then,
\[\abs{P_n(x, y)} = e_x^T M_{s, t}^n e_y,\]
where $e_x$ and $e_y$ are the coordinate vectors for the states $x$ and $y$, respectively.

For each $x \in \mc{X}_{s, t}$, let $\Gamma_t(x)$ be the collection of paths $x = x_1, \ldots, x_t$ such that for all $i = 1, 2, \ldots, t - 1$, $x_{i + 1} = T_{j_i}(x_i)$ for some $j_i \le t + 1 - i$. Then, by Proposition~\ref{prop:sequence_matching_bijection}, every matching in $\mc{M}_{n, x}$ corresponds to a path $\gamma = (x_1, \ldots, x_n)$ such that $(x_1, \ldots, x_{n - t + 1}) \in P_{n - t}(x, x_{n - t + 1})$ and $(x_{n - t + 1}, \ldots, x_n) \in \Gamma_t(x_{n - t + 1})$. Consequently,
\begin{equation}\label{eq:permanent_paths}
  \abs{A_{n, x}} = \sum_{z \in \mc{X}_{s, t}} \abs{P_{n - t + 1}(x, z)} \cdot \abs{\Gamma_t(z)} = \sum_{z \in \mc{X}_{s, t}} e_x^T M_{s, t}^{n - t + 1} e_z \cdot \abs{\Gamma_t(z)}.
\end{equation}
It therefore suffices to show that
\[\frac{e_x M_{s, t}^n e_z}{e_y M_{s, t}^n e_z} = \frac{v_x}{v_y} \left( 1 + O(e^{-n}) \right)\]
for all states $x, y, z$.

To this end, let $m = \abs{X_{s, t}}$. Suppose that for each $n > 0$, $M_{s, t}^n$ has the singular value decomposition $M_{s, t}^n = U_n \Sigma_n V_n^T$, where the diagonal entries $\sigma_{n, 1}, \ldots, \sigma_{n, m}$ of $\Sigma$ are arranged so that $\abs{\sigma_{n, 1}} \ge \cdots \ge \abs{\sigma_{n, m}}$. Denote the columns of $U$ and $V$ by $u_{n, 1}, \ldots, u_{n, m}$ and $v_{n, 1}, \ldots, v_{n, m}$, respectively. Both sets of vectors form orthonormal bases of $\R^m$ and are related by
\begin{align*}
  M_{s, t}^n v_{n, i} &= \sigma_{n, i} u_{n, i} \\
  u_{n, i} M_{s, t}^n &= \sigma_{n, i} v_{n, i}.
\end{align*}
It is shown in~\cite{yam67} that for all $i$,
\begin{equation}\label{eq:singular_matrix_power}
  \lim_{n \rightarrow \infty} \sigma_{n, i}^{1/n} = \lambda_i,
\end{equation}
where $\abs{\lambda_1} > \abs{\lambda_2} \ge \cdots \ge \abs{\lambda_m}$ are the eigenvalues of $M_{s, t}$. Letting $\delta = \frac{1}{2} \left( 1 - \frac{\abs{\lambda_2}}{\abs{\lambda_1}} \right)$, there must exist $N_\delta > 0$ such that for all $n > N_\delta$ and $i = 2, \ldots, m$, $\frac{\sigma_{n, i}^{1/n}}{\lambda_1} < 1 - \delta$.

For each $1 \le i \le m$, let $c_{n, i} = v \cdot v_{n, i}$, so that
\[v = \sum_{i = 1}^m c_{n, i} v_{n, i}.\]
Multiplying by $M_{s, t}^n$ yields
\begin{align}
  \lambda_1^n v &= M_{s, t}^n v \nonumber \\
  &= \sum_{i = 1}^m c_{n, i} M_{s, t}^n v_{n, i} \nonumber \\
  &= \sum_{i = 1}^m \sigma_{n, i} c_{n, i} u_{n, i}, \label{eq:v_decomp}
\end{align}
Dividing by $\lambda_1^n$ gives
\[v = \frac{\sigma_{n, 1} c_{n, 1}}{\lambda_1^n} u_{n, 1} + \sum_{i = 2}^m \frac{\sigma_{n, i} c_{n, i}}{\lambda_1^n} u_{n, i} = \frac{\sigma_{n, 1} c_{n, 1}}{\lambda_1^n} u_{n, 1} + u'_n,\]
where $u'$ is a vector of norm at most $m(1 - \delta)^n \le m e^{-n \delta}$. Since both $v$ and $u_{n, 1}$ are unit vectors, it follows that for any $z \in \mc{X}_{s, t}$,
\begin{equation}\label{eq:right_eigenvector_convergence}
  \abs{\frac{v \cdot e_z}{u_{n, 1} \cdot e_z} - 1} \le m e^{-n \delta}.
\end{equation}

A similar argument shows that
\begin{equation}\label{eq:left_eigenvector_convergence}
  \abs{\frac{w \cdot e_z}{v_{n, 1} \cdot e_z} - 1} \le m e^{-n \delta},
\end{equation}
where $w$ is the left eigenvector corresponding to $\lambda_1$ and is chosen to have all positive coordinates.

Finally, consider the decompositions
\begin{align*}
  e_x &= \sum_{i = 1}^m a_{n, i} u_{n, i} \\
  e_y &= \sum_{i = 1}^m b_{n, i} u_{n, i},
\end{align*}
where by~\eqref{eq:right_eigenvector_convergence},
\begin{align*}
  a_{n, i} &= e_x \cdot u_{n, i} = v_x \left( 1 + O(e^{-n}) \right), \\
  b_{n, i} &= e_y \cdot u_{n, i} = v_y \left( 1 + O(e^{-n}) \right).
\end{align*}

The $z$ coordinate of $e_x M_{s, t}^n$ then satisfies
\begin{align*}
  \frac{e_x M_{s, t}^n e_z}{\lambda_1^n} &= \frac{1}{\lambda_1^n}  \sum_{i = 1}^m a_{n, i} u_{n, i} M_{s, t}^n e_z  \\
  &= \frac{1}{\lambda_1^n} \sum_{i = 1}^m a_{n, i} \sigma_{n, i} v_{n, i} \cdot e_z \label{eq:z_coord_decomp} \\
  &= \frac{a_{n, 1} \sigma_{n, 1} v_{n, 1} \cdot e_z}{\lambda_1^n} + u'' \\
  &= \frac{a_{n, 1}\sigma_{n, 1} w_z}{\lambda_1^n} + u'',
\end{align*}
where $u''$ is a vector of length at most $m e^{-n\delta}$.
Similarly, $\frac{e_y M_{s, t}^n e_z}{\lambda_1^n} = \frac{b_{n, 1} \sigma_{n, 1} w_z}{\lambda_1^n} + u'''$.
Putting everything together therefore yields
\[\frac{e_x M_{s, t}^n e_z}{e_y M_{s, t}^n e_z} = \frac{e_x \cdot v}{e_y \cdot v} \left( 1 + O(e^{-n}) \right) = \frac{v_x}{v_y} \left( 1 + O(e^{-n}) \right),\]
thus completing the proof.

\noindent The proofs of Lemmas~\ref{lem:conditional_dist} and~\ref{lem:var_bound} are due to Sourav Chatterjee.
\subsection{Proof of Lemma~\ref{lem:conditional_dist}}
For each $0 \le k \le n - 1$ and each $x, y \in \{0, 1, \ldots, t\}$, define
\[m_k(x, y) = \Pb(X_{k + 1} = y \mid X_k = x).\]
Take any $x_1, \ldots, x_n \in \{0, 1, \ldots, t\}$. Define $x_0 = x_{n + 1} = 0$ and $X_{n + 1}^i = 0$. Let $m_n(x, 0) = 1$ for any $x \in \{0, 1, \ldots, t\}$. Let $z_i = 1 - x_i$ for each $i$. For any $x \in \{0, 1, \ldots, t\}$, define the event
\[E = \{X_1 = x_1, \ldots, X_{i + 1} = x_{i + 1}, X_{i + 1}^i = x\}.\]
There are two cases to consider. Furst, suppose that $x = x_{i + 1}$. In this case, if $E$ happens, then $\tau_i = i + 1$, and hence $(Y_{i + 2}^i, \ldots, Y_n^i) = (X_{i + 2}, \ldots, X_n)$, meaning the conditional distributions of $(Y_{i + 2}^i, \ldots, Y_n^i)$ and $(X_{i + 2}^i, \ldots, X_n^i)$ given $E$ are the same. \\
Next, suppose that $x \ne x_{i + 1}$. Then $\tau_i \ge i + 2$, and hence
\begin{align*}
  \Pb(Y_{i + 2}^i = x_{i + 2}, \ldots, Y_n^i = x_n \mid E) &= \sum_{j = i + 2}^{n + 1} \Pb(\{Y_{i + 2}^i = x_{i + 2}, \ldots, Y_n^i = x_n\} \cap \{\tau_i = j\} \mid E) \\
  &= \sum_{j = i + 2}^{n + 1} \Pb(A_j \cap B_j \mid E),
\end{align*}
where
\begin{align*}
  A_j &:= \{X_{i + 2}^i = x_{i + 2}, \ldots, X_j^i = x_j\} \\
  B_j &:= \{X_{i + 2} \ne x_{i + 2}, \ldots, X_{j - 1} \ne x_{j - 1}, X_j = x_j \ldots, X_n = x_n\} \\
  &= \bigcup_{z_k \ne x_k \forall i + 2 \le k < j} \{X_{i + 2} = z_{i + 2}, \ldots, X_{j - 1} = z_{j - 1}, X_j = x_j, \ldots, X_n = x_n\}
\end{align*}
Now, since $(X_{i + 2}^i, \ldots, X_n^i)$ are $(X_{i + 2}, \ldots, X_n)$ are conditionally independent given $(X_1, \ldots, X_{i + 1}, X_{i + 1}^i)$,
\[\Pb(A_j \cap B_j \mid E) = \Pb(A_j \mid E) \Pb(B_j \mid E).\]
By the Markov property,
\begin{align*}
  \Pb(A_j \mid E) &= m_{i + 1}(x, x_{i + 2}) \prod_{k = i + 2}^{j - 1} m_k(x_k, x_{k + 1}) \\
  \Pb(B_j \mid E) &= \sum_{z_k \ne x_k \forall i + 2 \le k \le j} \left[ m_{i + 1}(x_{i + 1}, z_{i + 2}) \left( \prod_{l = i + 2}^{j - 2} m_l(z_l, z_{l + 1}) \right) \right. \\
  &\qquad\left. \cdot m_{j - 1}(z_{j - 1}, x_j) \left( \prod_{l = j}^{n - 1} m_l(x_l, x_{l + 1}) \right) \right] \\
  \Pb(B_j \mid E) &= \sum_{z_k \ne x_k \forall i + 2 \le k \le j} \left[ m_{i + 1}(x_{i + 1}, z_{i + 2}) \left( \prod_{l = i + 2}^{n - 1} m_l(x_l, x_{l + 1}) \right) \right]
\end{align*}
The product $\Pb(A_j \mid E) \Pb(B_j \mid E) = PQ_j$, where
\begin{align*}
  P &= m_{i + 1}(x, x_{i + 2}) \prod_{l = i + 2}^{n - 1} m_l(x_l, x_{l + 1}) \\
  Q_j &= \sum_{z_k \ne x_k \forall i + 2 \le k \le j} m_{i + 1}(x_{i + 1}, z_{i + 2}) m_{j - 1}(z_{j - 1}, x_j) \prod_{l = i + 2}^{j - 2} m_l(z_l, z_{l + 1})
\end{align*}
when $j \ge i + 3$, and $Q_{i + 2} = m_{i + 1}(x, x_{i + 2})$. But by the Markov property,
\begin{align*}
  P &= \Pb(X_{i + 2}^i = x_{i + 2}, \ldots, X_n^i = x_n \mid E) \\
  Q_j &= \sum_{z_k \ne x_k \forall i + 2 \le k \le j} \Pb(X_{i + 2} = z_{i + 2}, \ldots, X_{j - 1} = z_{j - 1}, X_j = x_j \mid E) \\
  &= \Pb(X_{i + 2} \ne x_{i + 2}, \ldots, X_{j - 1} \ne x_{j - 1}, X_j = x_j \mid E)
\end{align*}
Thus,
\[\Pb(Y_{i + 1}^i = x_{i + 1}, \ldots, Y_n^i = x_n \mid E) = P\sum_{j = i + 2}^{n + 1} Q_j.\]
Next, observe that the $Q_j$'s are conditional probabilities of disjoint events whose union is the whole sample space. Thus,
\[\sum_{j = i + 2}^{n + 1} Q_j = 1,\]
proving the first claim of the lemma. To prove the second claim, simply note that $Y_{i + 1}^i = X_{i + 1}^i$ when $\tau_i > i + 1$, and $Y_{i + 1}^i = X_{i + 1} = X_{i + 1}^i$ when $\tau_i = i + 1$.
\subsection{Proof of Lemma~\ref{lem:var_bound}}
For $0 \le i \le n$, define
\[f_i(x_1, \ldots, x_i) = \E(f(X) \mid X_1 = x_1, \ldots, X_i = x_i).\]
Then by the martingale decomposition of variance,
\[\Var(f(X)) = \sum_{i = 0}^{n - 1} \E(f_{i + 1}(X_1, \ldots, X_{i + 1}) - f_i(X_1, \ldots, X_i))^2.\]
Now note that
\begin{align*}
  &\E(f_{i + 1}(X_1, \ldots, X_{i + 1}) - f_i(X_1, \ldots, X_i))^2 \\
  &\qquad= \E(\Var(f_{i + 1}(X_1, \ldots, X_{i + 1}) \mid X_1, \ldots, X_i)) \\
  &\qquad= \frac{1}{2} \E\left( \E(f_{i + 1}(X_1, \ldots, X_{i + 1}) - f_i(X_1, \ldots, X_i))^2 \mid X_1, \ldots, X_i \right) \\
  &\qquad= \frac{1}{2} \E\left( f_{i + 1}(X_1, \ldots, X_{i + 1}) - f_{i + 1}(X_1, \ldots, X_i, X_{i + 1}^i) \right)^2,
\end{align*}
where the second identity holds since $X_{i + 1}$ and $X_{i + 1}^i$ are i.i.d.\ conditional on $X_1, \ldots, X_i$. Now,
\begin{align*}
  f_{i + 1}(X_1, \ldots, X_{i + 1}) &= \E(f(X_1, \ldots, X_n) \mid X_1, \ldots, X_{i + 1}) \\
  &= \E(f(X_1, \ldots, X_n) \mid X_1, \ldots, X_{i + 1}, X_{i + 1}^i) \\
  f_{i + 1}(X_1, \ldots, X_i, X_{i + 1}^i) &= \E(f(X_1, \ldots, X_i, X_{i + 1}^i, \ldots, X_n^i \mid X_1, \ldots, X_i, X_{i + 1}^i)) \\
  &= \E(f(X_1, \ldots, X_i, X_{i + 1}^i, \ldots, X_n^i) \mid X_1, \ldots, X_i, X_{i + 1}, X_{i + 1}^i).
\end{align*}
By Lemma~\ref{lem:conditional_dist},
\begin{align*}
  &\E(f(X_1, \ldots, X_i, X_{i + 1}^i, \ldots, X_n^i) \mid X_1, \ldots, X_{i + 1}, X_{i + 1}^i) \\
  &\qquad= \E(f(X_1, \ldots, X_i, X_{i + 1}^i, Y_{i + 2}^i, \ldots, Y_n^i) \mid X_1, \ldots, X_{i + 1}, X_{i + 1}^i) \\
  &\qquad= \E(f(X_1, \ldots, X_i, Y_{i + 1}^i, \ldots, Y_n^i) \mid X_1, \ldots, X_{i + 1}, X_{i + 1}^i).
\end{align*}
Putting everything together yields
\begin{align*}
  &\E(f_{i + 1}(X_1, \ldots, X_{i + 1}) - f_{i + 1}(X_1, \ldots, X_i, X_{i + 1}^i))^2 \\
  &\qquad= \E(\E(f(X_1, \ldots, X_n) - f(X_1, \ldots, X_i, Y_{i + 1}^i, \ldots, Y_n^i) \mid X_1, \ldots, X_{i + 1}, X_{i + 1}^i))^2 \\
  &\qquad\le \E(f(X_1, \ldots, X_n) - f(X_1, \ldots, X_i, Y_{i + 1}^i, \ldots, Y_n^i))^2.
\end{align*}
Therefore,
\[\Var f(X) \le \frac{1}{2} \sum_{i = 0}^{n - 1} \E(f(X_1, \ldots, X_n) - f(X_1, \ldots, X_i, Y_{i + 1}^i, \ldots, Y_n^i))^2,\]
as desired.

\section{Acknowledgements}
\noindent This research was funded by NSF DMS Grant 1501767.

The author would like to thank Persi Diaconis, Sourav Chatterjee, and Brett Kolesnik for helpful conversations and encouragement.

\end{document}